\documentclass{amsart}

\usepackage{amssymb}
\usepackage{amscd}
\usepackage[all,cmtip]{xy}
\usepackage{hyperref}
\usepackage{graphicx}

\newcommand{\comment}[1]{}

\newcommand{\R}{\mathbf{R}}
\newcommand{\Z}{\mathbf{Z}}

\newcommand{\kv}{\mathrm{k}}

\DeclareMathOperator{\Gal}{Gal}

\DeclareMathOperator{\ord}{ord}
\DeclareMathOperator{\fv}{f}

\DeclareMathOperator{\Dv}{D}
\DeclareMathOperator{\Iv}{I}

\theoremstyle{plain}
\newtheorem{theorem}{Theorem}[section]

\newtheorem{corollary}[theorem]{Corollary}

\newtheorem{lemma}[theorem]{Lemma}

\theoremstyle{definition}

\newtheorem{example}[theorem]{Example}

\begin{document}

\title{A short proof of a Chebotarev density theorem for function fields}
\author{Michiel Kosters}
\address{University of California, Irvine, Department of
Mathematics, 340 Rowland Hall, Irvine, CA 92697}
\email{kosters@gmail.com}
\urladdr{https://sites.google.com/site/kosters}
\date{\today}
\thanks{This is part of my PhD thesis written under the supervision of Hendrik Lenstra.}

\begin{abstract}
In this article we discuss a version of the Chebotarev density for function fields over perfect fields with procyclic absolute Galois groups. Our version of this density theorem differs from other versions in two aspects: we include ramified primes and we do not have an error term.
\end{abstract}

\maketitle

\section{Introduction}

\subsection{Motivation}

One of the important results in arithmetic geometry is called `the' Chebotarev density theorem for function fields. We will first briefly describe such a theorem. Let $k$ be a finite field and let $K$ be a function field over $k$. Let $M/K$ be a finite Galois extension with group $G$. To a prime $P$ of $K$ which is unramified in $M/K$, one can associate a conjugacy class $(P,M/K)$ of $G$, called the Frobenius class. This Frobenius element for example gives the splitting behavior of $P$ in $M/K$. The Chebotarev density theorem, in many different forms, gives an equidistribution result for the occurrence of conjugacy classes as the Frobenius class of primes. An example of such a theorem is \cite[Theorem 9.13B]{ROS}:

\begin{theorem}
Let $M/K$ be a finite Galois extension of function fields over a finite field $k$ of cardinality $q$ with group $G$. Assume that $k$ is the full constant field of $M$. Let $C \subseteq G$ be a conjugacy class and let $S'_K$ be the set of primes in $K$ which are unramified in $M$. Then for each positive integer $n$ we have
\begin{align*}
\# \{ P \in S_K' | \deg_k(P)=n,\ (P,M/K)=C \} = \frac{\#C}{\#G} \frac{q^n}{n} + O \left( \frac{q^{n/2}}{n} \right).
\end{align*}
\end{theorem}

We will state a different version of a density theorem which is significantly different from the above theorem. Most importantly, our version will include ramified primes. Secondly, our version will not contain an error term. Instead of an equidistribution result, we `parametrize' the points with a given Frobenius class by primes of twists of our original function field. We can get a version with error terms by taking into account the genus of the auxiliary function fields. Furthermore, our statement is more general than the above statement, although these generalities were known. We allow $k$ to be perfect with procyclic absolute Galois group instead of being a finite field, and $k$ does not necessarily need to be the full constant field of $M$. For simplicity however, we restrict to the case when $n=1$, that is, we only treat rational primes.

\subsection{New Chebotarev density theorem} \label{444}

Let us describe the new Chebotarev density theorem.

Let $k$ be a perfect field with procyclic absolute Galois group with $F \in \Gal(\overline{k}/k)$ as a topological generator. Let $r=\prod_p p^{n_p}$ with $p$ prime and $n_p \in \Z_{\geq 0} \sqcup \{\infty\}$ be the order of the profinite group $\Gal(\overline{k}/k)$, which is a Steinitz number. For example, one can take $k$ to be a quasi-finite field, such as a finite field. There are also quasi-finite fields in characteristic $0$ \cite[Chapter XIII, Section 2]{SE1}. One can also take $k$ to be the maximal $p$-power extension of a finite field, or take $k=\R$. We can also allow $k$ to be algebraically closed, but the results below will not be interesting in that case. 

Let $K$ be a geometrically irreducible function field over $k$, that is, a finitely generated field extension of $k$ of transcendence degree $1$ such that $k$ is integrally closed in $K$. We denote by $\mathcal{P}_{K/k}$ the set of valuation rings of $K$ containing $k$ which are not equal to $K$. Let $P \in \mathcal{P}_{K/k}$. By $\kv_P$ we denote the residue field of the valuation ring $P$. We set $\deg_k(P)=[\kv_P:k]$. The subset of these valuation rings $P$ such that $\kv_P=k$, the set of rational primes of $K$, is denoted by $\mathcal{P}_{K/k}^1$. The restriction of $P$ to a subfield $K'$ of $K$ is denoted by $P|_{K'}$ and we say that $P$ lies above $P|_{K'}$. 

Let $M/K$ be a finite Galois extension with Galois group $G=\Gal(M/K)$. Let $P \in \mathcal{P}_{K/k}$ with valuation $Q$ above it in $M$. The group $G$ acts transitively on the set of primes above $P$. Set $\Dv_{Q,K}=\{g \in G: gQ=Q\}$ (\emph{decomposition group}). Note that we have a natural map $\Dv_{Q,K} \to \Gal_{\kv_P}(\kv_Q)$. The kernel of this map is called the \emph{inertia group} and is denoted by $\Iv_{Q,K}$. In fact, we have an exact sequence $1 \to \Iv_{Q,K} \to \Dv_{Q,K} \to \Gal(\kv_Q/\kv_P) \to 0$ (\cite[Theorem 3.6]{KO7}). After a choice of a $k$-embedding $\kv_Q \subseteq \overline{k}$ we have a map $\Gal(\overline{k}/\kv_P) \to \Gal(\kv_Q/\kv_P)$. The image of $F^{[\kv_P:k]}$ is a generator of $\Gal(\kv_Q/\kv_P)$ which does not depend on the choice of the embedding. The set of elements in $\Dv_{Q,K}$ mapping to this generator is denoted by $(Q,M/K)$. We set 
\begin{align*}
(P,M/K)=\{g x g^{-1}: g \in G, x \in (Q,M/K)\}=\bigcup_{Q' \in \mathcal{P}_{M/k}:\ Q'|_K=P} (Q',M/K),
\end{align*}
which is a union of conjugacy classes.  This definition does not depend on the choice of $Q$. If $\Iv_{Q,K}=0$, this is just a single conjugacy class.

We define a probability measure $(P,M)$ on $G$ as follows. This is a function $(P,M): G \to \R_{\geq 0}$ such that $\sum_{g \in G} (P,M)(g)=1$. For $\gamma \in G$, with
conjugacy class $\Gamma$, we set:\index{$(P,M)(\gamma)$}
\begin{align*}
 (P,M)(\gamma)= \frac{\# \left((Q,M/K) \cap \Gamma\right)}{\#\Gamma \cdot \#(Q,M/K)}
\end{align*}
for any prime $Q$ of $M$ above $P$.
If for $Q$ above $P$ one has $\Iv_{Q,K}=0$, then the distribution is evenly divided over the whole conjugacy class $(P,M/K)$ and zero outside. If $\Iv_{Q,K}=0$ and $\gamma \in G$ with $\ord(\gamma) \nmid r$ (divisibility of Steinitz numbers), one has $(P,M)(\gamma)=0$.

Let $k'$ be the integral closure of $k$ in $M$ (the full constant field of $M$). Let $N=\Gal(M/Kk')$, which is the geometric Galois group. Note that $G/N
=\Gal(Kk'/K)=\Gal(k'/k)=\langle \overline{F} \rangle$, where $\overline{F}$ is the image of $F$ under $\Gal(\overline{k}/k) \to \Gal(k'/k)$. We view $\overline{F}$ as an element of $G/N$, hence as a coset of $G$. If $P$ is a rational prime of $K$, one easily finds $(P,M/K) \subseteq \overline{F} \subseteq G$. 

We have the following alternative version of the Chebotarev density theorem.

\begin{theorem} \label{25c675}
Assume that we are in the situation as described above. Let $\gamma \in \overline{F} \subseteq G$ and assume that $m=\ord(\gamma)|r$. Let $k_m$ be the unique extension of degree $m$ of $k$ in some algebraic closure of $K$
containing $M$. Then the following hold:
\begin{enumerate}
 \item There is a unique element $\Delta \in \Gal(k_mM/K)$ such that $\Delta|_M=\gamma$ and $\Delta|_{k_m}=F|_{k_m}$;
 \item $M_{\gamma}=\left(k_mM\right)^{\langle \Delta \rangle}$ is geometrically irreducible over $k$ and satisfies $k_m M_{\gamma}=k_mM$.
\end{enumerate}
Let 
\begin{align*}
 \phi \colon \mathcal{P}_{M_{\gamma}/k}^1 &\to  \mathcal{P}_{K/k}^1 \\
 Q &\mapsto Q|_K
\end{align*}
be the natural map. For $P \in \mathcal{P}_{K/k}^1$ we have $\# \phi^{-1}(P)= \#N \cdot
(P,M)(\gamma)$. Finally, we have
\begin{align*}
\sum_{P \in \mathcal{P}_{K/k}^1} (P,M)(\gamma)=\frac{\# \mathcal{P}_{M_{\gamma}/k}^1}{\# N}.
\end{align*}
\end{theorem}

Note that $M_{e}=M$, where $e$ is the identity element of $G$.

To see the resemblance with the conventional versions of the Chebotarev density theorem, we consider the case when $k$ is a finite field. We denote the genus of $M$ by $g_{k'}(M)$. The genus of $M$, which is equal to the genus of $M_{\gamma}$, allows us to estimate $\# \mathcal{P}_{M_{\gamma}/k}^1$. We find the following. 

\begin{corollary} \label{25c6}
Assume that $k$ is finite of cardinality $q$. Let $\gamma \in \overline{F}$. Then we have
\begin{align*}
 |\sum_{P \in \mathcal{P}^1_{K/k}} (P,M)(\gamma)-\frac{1}{\#N}(q+1)| \leq \frac{1}{\#N} 2g_{k'}(M) \sqrt{q}.
\end{align*}
\end{corollary}
To get an even more familiar version, sum over all $\gamma$ inside a conjugacy class. A similar approach can be found in \cite[Section 6.4]{JAR}, especially in Lemma 6.4.2, Lemma 6.4.4 and Proposition 6.4.8. In this approach one does not consider ramified primes and hence the above result is modified, by estimating the number of ramified points using Riemann-Hurwitz. Hence the expression they obtain involves more terms. Also, a similar result can be obtained when $M/K$ is just assumed to be normal, since primes extend uniquely in purely inseparable extensions. We finish this section by giving an example of our approach, which shows how the ramified primes make the results nicer.

\begin{example}
Let $k$ be a finite field of cardinality $q$. Let 
\begin{align*}
E: Y^2 + a_1XY+a_3Y=X^3+a_2X^2+a_4X+a_6
\end{align*}
 be an elliptic curve over $k$. Let $M=k(X)[Y]/(Y^2 + a_1XY+\ldots)$ be the function field of $E$, which is a cyclic extension of degree $2$ over $K=k(X)$ with group say $G=\{e, g\}$. The rational primes which are ramified in the extension $M/K$ correspond to the $X$-coordinates of the points of $E(k)[2]$. The unramified rational primes which split, have $e$ as Frobenius. The unramified primes which do not split, have $g$ as Frobenius. The ramified rational primes have Frobenius equally split over $e$ and $g$. One has $N=G$ and we find: 
\begin{align*}
\sum_{P \in \mathcal{P}^1_{K/k}} (P,M)(e)= \frac{\# E(k)- \# E(k)[2]}{2} +  \frac{\# E(k)[2]}{2}= \# E(k)/2.
\end{align*}
Using $\sum_{h \in G} \sum_{P \in \mathcal{P}^1_{K/k}} (P,M)(h) = q+1$, one finds
\begin{align*}
\sum_{P \in \mathcal{P}^1_{K/k}} (P,M)(g)= (q+1)-\# E(k)/2.
\end{align*}
One has $M_e=M$, and $M_g$ is the function field of the quadratic twist of $E$. Note that the number of unramified rational primes with $e$ as Frobenius has the less clean expression $\frac{\# E(k)- \# E(k)[2]}{2}$.
\end{example}

\subsection{Proof of the Chebotarev density theorem}

We continue using the notation introduced before Theorem \ref{25c675}. Set $h=[k':k]$. 

\begin{lemma} \label{25c55}
Assume that we are in the situation as described above. Let $\gamma \in \overline{F} \subseteq G$ and assume that $m=\ord(\gamma)|r$. Let $k_m$ be the unique extension of degree $m$ of $k$ in some algebraic closure of $K$
containing $M$. Then the following hold:
\begin{enumerate}
 \item There is a unique element $\Delta \in \Gal(k_mM/K)$ such that $\Delta|_M=\gamma$ and $\Delta|_{k_m}=F|_{k_m}$;
 \item $M_{\gamma}=\left(k_mM\right)^{\langle \Delta \rangle}$ is geometrically irreducible over $k$ and satisfies $k_m M_{\gamma}=k_mM$.
\end{enumerate}
Furthermore, there is a well-defined map
\begin{align*}
\varphi \colon \mathcal{P}_{M_{\gamma}/k}^1 &\to S=\{Q \in \mathcal{P}_{M/k}: \deg_k(Q|_K)=1,\ \gamma \in (Q,M/K) \} \\
Q'|_{M_{\gamma}} &\mapsto Q'|_M,
\end{align*}
where $Q' \in \mathcal{P}_{k_mM/k}$. For $Q \in S$ we have $\# \varphi^{-1}(Q) = \frac{\deg_k(Q)}{h}$.
\end{lemma}
\begin{proof} Note that $\gamma N$ is a generator in $G/N$, and hence one finds $m \equiv 0
\pmod{\#G/N}$. This shows that $k_mK \cap M = k'K$.
We have the following diagram:

\[
\xymatrix{
& k_mM & \\
k_m K \ar[ru] & & M \ar[lu] \\
& k'K=k_mK \cap M \ar[lu] \ar[ru] & \\
& K. \ar[u] &
}  
\]
One has $\Gal(k_mM/K)=\Gal(k_mK/K) \times_{\Gal(k'K/K)} \Gal(M/K) \ni \Delta$ and the first statement follows.
Note that $M_{\gamma}\cap k_{m}K=K$. Furthermore, notice that $k_mM_{\gamma}$ is not fixed by any non-trivial element of $\langle \Delta \rangle$, and hence is equal to $k_mM$. 

We claim that the following three statements are equivalent for $P' \in \mathcal{P}_{k_mM/k}$:
1. $\gamma \in (P'|_M,M/K)$ and $P'|_K$ is rational; \\
2. $\Delta \in (P',k_mM/K)$ and $P'|_K$ is rational; \\
3. $P'|_{M_{\gamma}}$ is rational. 

1 $\implies$ 2: We claim $(P',k_mM/K)=(P'|_{k_mK},k_mK/K) \times (P'|_M,M/K)$. Indeed, both sets have the same size as $k_mK/K$ is unramified and it follows that the natural injective map is a bijection. From the rationality of $P'|_K$ and the assumption $\gamma \in (P'|_M,M/K)$, one obtains $\Delta \in (P',k_mM/K)$.

2 $\implies$ 1: obvious. 

2 $\implies$ 3: By 2 one has $\Gal(k_mM/M_{\gamma})=\langle \Delta  \rangle = \Dv_{P',M_{\gamma}}$. It follows that in the constant field extension $k_mM/M_{\gamma}$ that $P'$ is the unique prime above $P'_{M_{\gamma}}$. From 2 it follows that $\deg_k(P')|\mathrm{ord}(\Delta)=m$. By the theory of constant field extensions \cite[Theorem 3.6.3]{ST}, it follows that $P'_{M_{\gamma}}$ is rational. 

3 $\implies$ 2: If $P'|_{M_{\gamma}}$ is rational, then so is $P'|_K$, and one finds
\begin{align*}
\{\Delta \} = (P', k_mM/M_{\gamma}) \subset (P', k_mM/K).
\end{align*}

We will now look at $\varphi$. For a rational prime $P \in \mathcal{P}_{M_{\gamma}/k}^1$ there is a unique prime above it in $k_mM$. The implication 3 $\implies$ 1 shows that $\varphi$ is well-defined. We will calculate the sizes of the fibers.
Take a prime $Q \in S$. Notice that $[k_mM:M]=m/h$ and that $\deg_k(Q)|m$. In the extension
$k_mM/M$, the residue field of $Q$ grows with a degree $m/\deg_k(Q)$ and hence there are $\frac{m/h}{m/\deg_k(Q)}=\frac{\deg_k(Q)}{h}$
primes above it in $k_mM$ \cite[Theorem 3.6.3]{ST}. Each such prime gives a different preimage under $\varphi$ and by $1 \implies$ 3 we find all preimages.
\end{proof}

\begin{lemma} \label{25c56}
 Let $\gamma \in \overline{F}$ and let $\Gamma$ be its conjugacy class in $G$. Consider the natural surjective restriction map, where $S=\{Q \in \mathcal{P}_{M/k}: \deg_k(Q|_K)=1,\ \gamma \in (Q,M/K) \}$ and $T=\{P \in \mathcal{P}_{K/k}: \deg_k(P)=1, \gamma \in (P,M/K)\}$,
\begin{align*}
 \psi \colon S \to T,\ Q \mapsto Q|_K.
\end{align*}
Then for $P \in T$ with prime $Q \in S$ above it we have 
\begin{align*}
 \# \psi^{-1}(P)=\frac{\#G}{\#\Gamma \cdot \# \Dv_{Q,K}} \cdot \#\left( (Q,M/K) \cap \Gamma\right).
\end{align*}
\end{lemma}
\begin{proof}
 Let $Q \in S$ lie above $P$. Then we have $(Q,M/K)=\gamma \Iv_{Q,K}$. For $g \in G$ we have $(gQ,M/K)=g(Q,M/K)g^{-1}$. So $\gamma \in (gQ,M/K)$ iff
$\gamma \in g (Q,M/K) g^{-1}$ iff $g^{-1} \gamma g \in (Q,M/K)$. Let $G_{\gamma}$ be the stabilizer of $\gamma$ under the conjugation action of
$G$ on itself. Then the number of $g \in G$ such that $\gamma \in (gQ,M/K)$ is equal to $\# G_{\gamma} \cdot
\# \left((Q,M/K) \cap \Gamma \right)=\frac{\#G}{\#\Gamma} \cdot  \# \left((Q,M/K) \cap \Gamma \right)$. 

Furthermore, suppose that for $g, g' \in G$ we have $gQ=g'Q$. Then $g'^{-1}g \in \Dv_{Q,K}$. This shows that $\# \psi^{-1}(P)=\frac{\#G}{\#\Gamma
\cdot \# \Dv_{Q,K}} \cdot \# \left((Q,M/K) \cap \Gamma \right)$.
\end{proof}

We can finally prove the new version of the Chebotarev density theorem.

\begin{proof}[Proof of Theorem \ref{25c675}]
The first part directly follows from Lemma \ref{25c55}. The rest of the proof will follow from combining Lemma \ref{25c55} and Lemma \ref{25c56}. We follow the notation from these lemmas. Note that $\phi=\psi \circ \varphi$. Let $P \in T$ and let $Q \in \psi^{-1}(P)$. Note that $\deg_k(Q)$ does not depend on the choice of $Q$. 
One has:
\begin{align*}
\# \phi^{-1}(P) &= \# \varphi^{-1} \circ \psi^{-1}(P) \\
&= \frac{\deg_k(Q)}{h} \cdot \frac{\#G}{\#\Gamma \cdot \# \Dv_{Q,K}} \cdot \#\left( (Q,M/K) \cap \Gamma\right) \\
&= \frac{\#N}{\#\Gamma} \cdot \frac{\deg_k(Q) \cdot \#\left( (Q,M/K) \cap \Gamma\right)}{\# \Dv_{Q,K}} \cdot \frac{\#(Q,M/K)}{\#(Q,M/K)} \\
&= \#N \cdot (P,M)(\gamma) \cdot \deg_k(Q) \cdot \frac{\#(Q,M/K)}{\# \Dv_{Q,K}} \\
&= \#N \cdot (P,M)(\gamma).
\end{align*}
The last statement follows very easily.
\end{proof}

We will now prove the corollary.

\begin{proof}[Proof of Corollary \ref{25c6}] We use Theorem \ref{25c675}. By Hasse-Weil (\cite[Theorem 5.2.3]{ST}) we have 
\begin{align*}
|\{P \in \mathcal{P}_{M_{\gamma}/k}: \deg_k(P)=1\}-(q+1)\}| \leq 2g_k(M_{\gamma}) \sqrt{q}=2g_{k'}(M) \sqrt{q}. 
\end{align*}
This gives the required result.
\end{proof}

\bibliographystyle{acm}

\end{document}